\documentclass[11pt]{article}
\usepackage[margin=1in]{geometry}
\usepackage{amsmath,amssymb,amsthm,mathtools}
\usepackage[hidelinks]{hyperref}
\usepackage{microtype}

\newtheorem{theorem}{Theorem}
\newtheorem{lemma}{Lemma}
\newtheorem{corollary}{Corollary}

\title{Edge-addition monotonicity of positive $p$-energy fails for every $p\ge1$}
\author{Koyar Afrasyab}
\date{}

\begin{document}
\maketitle

\begin{abstract}
For a graph $G$ with adjacency eigenvalues $\lambda_i$, define the positive $p$-energy by
\[
\mathcal E_p^+(G)=\sum_{\lambda_i>0}\lambda_i^p.
\]
At a 2021 AIM workshop, Guo conjectured that the positive square energy $s^+=\mathcal E_2^+$ should inherit the familiar edge-addition monotonicity of the spectral radius, $\rho(G+uv)\ge\rho(G)$.  That conjecture was subsequently shown to fail at $p=2$.  Tang, Liu, and Wang then introduced positive $p$-energy, proved nonmonotonicity for every $1\le p<3$, and in version 3 of their preprint (26 March 2025) explicitly conjectured that monotonicity should hold for $p\ge3$.  We disprove this conjectured high-exponent extension completely: for every real $p>2$ there are infinitely many connected graphs $G$ and nonedges $uv$ such that
\[
\mathcal E_p^+(G+uv)<\mathcal E_p^+(G).
\]
Together with the Tang--Liu--Wang counterexamples below $3$, this shows that no exponent $p\ge1$ restores the spectral-radius-style monotonicity: positive $p$-energy can decrease under the addition of an edge for every real $p\ge1$.

The construction is a chain of clique blocks joined by regular bipartite graphs. An equitable quotient converges to $Q=I+cA(P_k)$. For noninteger $p$, a binomial-series sign argument for a fractional power of $I-cA(P_k)$ gives the required negative endpoint entry. At integer exponents, choosing $c$ across the first spectral threshold leaves exactly one negative eigenvalue, and path locality forces the positive spectral contribution to have negative sign. We also give a fully rational $38$-vertex certificate at $p=4$ and determine the complete failure interval of a fixed $17$-vertex counterexample at $p=3$.
\end{abstract}

\section{Introduction}
For a simple graph $G$ and a real number $p\ge1$, let
\[
\mathcal E_p^+(G)=\sum_{\lambda_i(G)>0}\lambda_i(G)^p.
\]
The surrounding literature begins with classical graph energy and its Schatten-norm extensions; see, for example, Nikiforov \cite{Nikiforov2012,Nikiforov2016} and Arizmendi--Guerrero \cite{ArizmendiGuerrero2023}.  The separate positive and negative square energies were developed in connection with the Elphick--Farber--Goldberg--Wocjan conjecture \cite{ElphickEtAl2016}; subsequent work includes Abiad et al.\ \cite{AbiadEtAl2023}, Elphick--Linz \cite{ElphickLinz2024}, and the vertex-partitioning inequalities of Akbari et al.\ \cite{AkbariEtAl2025}.  The square-energy conjecture itself was recently resolved by Liu, Tang, and Zhang \cite{LiuTangZhang2026}.  Positive and negative $p$-energies have also appeared in recent extremal and coloring results \cite{ChenWangZhang2026,TangElphickZhang2026}.

A more direct lineage for the present problem comes from spectral-radius monotonicity.  If $uv\notin E(G)$, then Perron--Frobenius monotonicity gives $\rho(G+uv)\ge\rho(G)$.  At the 2021 AIM workshop on spectral graph and hypergraph theory, Guo proposed that the positive square energy should satisfy the analogous inequality
\[
s^+(G+uv)\ge s^+(G),
\]
explicitly as a strengthening of this spectral-radius phenomenon \cite{CioabaGuoSrivastava2021}.  The conjecture is false: Abiad et al. gave an explicit graph for which deleting one edge increases $s^+$ \cite[Example~2.6]{AbiadEtAl2023}, equivalently an edge addition that decreases it.

Tang, Liu, and Wang \cite{TangLiuWang2026} generalized the problem from $s^+=\mathcal E_2^+$ to positive $p$-energy.  They constructed counterexamples for every $1\le p<3$.  Importantly, version 3 of their arXiv preprint did more than pose an open question: after reporting that their constructions did not refute monotonicity for $p\ge3$, they explicitly reformulated Guo's conjecture as Conjecture~3, asserting that
\[
\mathcal E_p^+(G+uv)\ge\mathcal E_p^+(G)\qquad(p\ge3),
\]
for every graph and every missing edge \cite[Conjecture~3]{TangLiuWang2025v3}.  The published version phrases the remaining regime as an open monotonicity problem \cite{TangLiuWang2026}.  Thus the issue is a conjectured extension of spectral-radius monotonicity, not merely an isolated inequality question.

We disprove that conjectured extension completely.  The local monotonicity problem is distinct from the global extremal problem of minimizing $p$-energy or positive $p$-energy among connected graphs; recent path-minimality results are given in \cite{LiuTangPathEnergy2026,LiuTang2026}.

\begin{theorem}\label{thm:main}
For every real $p>2$, there are infinitely many connected simple graphs $G$ and nonedges $uv$ such that
\[
\mathcal E_p^+(G+uv)<\mathcal E_p^+(G).
\]
\end{theorem}

\begin{corollary}\label{cor:allp}
For every real $p\ge1$, positive $p$-energy is not monotone under edge addition.
\end{corollary}

\begin{proof}
For $1\le p<3$ this is the result of Tang--Liu--Wang \cite{TangLiuWang2026}. For $p\ge3$ it follows from Theorem \ref{thm:main}; the overlap $2<p<3$ is harmless.
\end{proof}

The key point is that the phenomenon has no upper threshold. In particular, $p=4$ is not a monotonicity boundary, and neither is any larger exponent. This is compatible with recent global path-minimality results for positive $p$-energy \cite{LiuTang2026}: a graph parameter can have a global minimum at a path while still decreasing under a particular local edge addition.

An earlier three-block family also gives counterexamples throughout $2<p<4$. To make the quantifiers explicit, that auxiliary argument has the form
\[
\forall p\in(2,4)\;\exists K=K(p)\;\exists N=N(p,K)\;\forall n\ge N:\quad\Delta_{n,K}(p)<0.
\]
Thus the block-ratio parameter $K$ is chosen \emph{after} fixing $p$; no single $K$ works uniformly all the way up to $4$. In fact the nonuniformity becomes acute as $p\uparrow4$ (rather than as $p\downarrow2$). The chain construction below makes that auxiliary family unnecessary for the main theorem.

\section{The chain construction and the quotient limit}
Fix an integer $k\ge2$. Let $C_1,\dots,C_k$ be pairwise disjoint cliques of order $n$. Between $C_i$ and $C_{i+1}$ place any $d$-regular bipartite graph, for $1\le i<k$, and place no edges between nonconsecutive blocks. Add two vertices $u,v$, join $u$ to every vertex of $C_1$, and join $v$ to every vertex of $C_k$. Let $H_{n,d,k}^{(0)}$ be the resulting graph, and set
\[
H_{n,d,k}^{(1)}=H_{n,d,k}^{(0)}+uv.
\]
For every integer $0\le d\le n$, a $d$-regular bipartite graph with two parts of size $n$ exists, for example as a union of $d$ edge-disjoint perfect matchings of $K_{n,n}$. For $d\ge1$ the graph $H_{n,d,k}^{(0)}$ is connected.

The partition
\[
\{u\},C_1,\dots,C_k,\{v\}
\]
is equitable (see, e.g., \cite{CvetkovicEtAl2010} for equitable partitions and quotient spectra). In the normalized block-constant basis, and with $t\in[0,1]$ temporarily denoting the weight of $uv$, the quotient matrix is
\begin{equation}
B_t(n,d)=
\begin{pmatrix}
0&\sqrt n e_1^\top&t\\
\sqrt n e_1&(n-1)I_k+dA(P_k)&\sqrt n e_k\\
t&\sqrt n e_k^\top&0
\end{pmatrix}.
\label{eq:Bt}
\end{equation}
The orthogonal complement of the block-constant subspace is invariant: a regular bipartite adjacency matrix sends zero-sum vectors to zero-sum vectors, clique adjacency acts as $-I$ there, and the endpoint vertices couple only to block-constant vectors. The restriction to this orthogonal complement is identical for $t=0$ and $t=1$. Hence it cancels from every energy difference.

For a real symmetric matrix $Q=\sum_j\theta_j w_jw_j^\top$ and a real $s$, define
\[
Q_+^s:=\sum_{\theta_j>0}\theta_j^s w_jw_j^\top.
\]
When $s=0$, this is the orthogonal projector onto the positive eigenspace.

\begin{lemma}[Quotient perturbation limit]\label{lem:limit}
Fix $k\ge2$.  Let $d=d_n$ satisfy $d_n/n\to c\in(0,1]$, and put
\[
Q=I_k+cA(P_k).
\]
Assume that $Q$ has no zero eigenvalue. For every fixed real $p>2$,
\begin{equation}
\lim_{n\to\infty}
\frac{\mathcal E_p^+(H_{n,d_n,k}^{(1)})-\mathcal E_p^+(H_{n,d_n,k}^{(0)})}{n^{p-2}}
=2p\,(Q_+^{p-3})_{1k}.
\label{eq:limit}
\end{equation}
\end{lemma}

\begin{proof}
Write $A=A(P_k)$ and
\[
Q_n=\left(1-\frac1n\right)I_k+\frac{d_n}{n}A,
\qquad
\widetilde B_t=\frac1n B_t(n,d_n),
\qquad
S_n=0\oplus Q_n\oplus0.
\]
All estimates below are uniform for $t\in[0,1]$; the implicit constants may depend on $k$ and $Q$, but not on $t$ or $n$.  From \eqref{eq:Bt}, using the Frobenius norm to bound the operator norm,
\begin{equation}
\sup_{0\le t\le1}\|\widetilde B_t-S_n\|
\le \sqrt{\frac4n+\frac2{n^2}}
\le \frac3{\sqrt n}.
\label{eq:uniform-matrix-pert}
\end{equation}
Moreover $Q_n\to Q$ in operator norm.  We shall use only the elementary Weyl eigenvalue bound and the Hellmann--Feynman identity for real symmetric matrices; see, e.g., Horn--Johnson \cite{HornJohnson2013}.

Let $\theta_1,\dots,\theta_k$ be the eigenvalues of $Q$.  They are simple because the eigenvalues of $A(P_k)$ are simple, and they are all nonzero by hypothesis.  Set
\[
\gamma=\frac14\min\left\{\min_j|\theta_j|,
\min_{i\ne j}|\theta_i-\theta_j|\right\}>0.
\]
For all sufficiently large $n$, the eigenvalues $\theta_{j,n}$ of $Q_n$ satisfy
$|\theta_{j,n}-\theta_j|<\gamma/2$, and the perturbation in
\eqref{eq:uniform-matrix-pert} is $<\gamma/2$.  Weyl's inequality therefore gives, for every $t\in[0,1]$, exactly one eigenvalue
$\mu_{j,t}(n)$ of $\widetilde B_t$ in a fixed $\gamma$-neighborhood of each $\theta_j$, while the remaining two eigenvalues lie in a neighborhood of $0$ whose radius tends to $0$.  In particular,
\begin{equation}
\sup_{t\in[0,1]}|\mu_{j,t}(n)-\theta_j|\longrightarrow0.
\label{eq:uniform-large-eigs}
\end{equation}
The separated eigenvalues $\mu_{j,t}(n)$ are simple for all sufficiently large $n$.

Fix $j$, and let $(x_t,y_t,z_t)$ be a unit eigenvector of $\widetilde B_t$ for $\mu_{j,t}(n)$, where $y_t\in\mathbb R^k$.  The endpoint and core equations are
\begin{align}
\mu_{j,t}x_t&=n^{-1/2}(y_t)_1+(t/n)z_t,\label{eq:scaled-end1}\\
(Q_n-\mu_{j,t}I)y_t&=-n^{-1/2}(x_te_1+z_te_k),\label{eq:scaled-core}\\
\mu_{j,t}z_t&=n^{-1/2}(y_t)_k+(t/n)x_t.\label{eq:scaled-end2}
\end{align}
By the definition of $\gamma$ and \eqref{eq:uniform-large-eigs},
$|\mu_{j,t}|\ge2\gamma$ for large $n$.  Equations
\eqref{eq:scaled-end1} and \eqref{eq:scaled-end2} thus give
\begin{equation}
|x_t|+|z_t|=O(n^{-1/2})
\label{eq:endpoint-small}
\end{equation}
uniformly in $t$.

The matrices $Q_n$ and $Q$ are affine functions of $A(P_k)$ and hence have the same orthonormal eigenvectors; write $w_j$ for a unit eigenvector belonging to $\theta_j$.  By \eqref{eq:scaled-core} and \eqref{eq:endpoint-small}, the residual of $y_t$ in the core equation is $O(n^{-1})$.  Expanding $y_t$ in the common eigenbasis of $Q_n$ and $Q$, the spectral separation by $\gamma$ implies that every component orthogonal to $w_j$ is $O(n^{-1})$, uniformly in $t$.  Since \eqref{eq:endpoint-small} and unit normalization give $\|y_t\|=1+O(n^{-1})$, we obtain
\begin{equation}
y_t=\sigma_{j,t,n}w_j+O(n^{-1}),
\qquad |\sigma_{j,t,n}|=1+O(n^{-1}),
\label{eq:uniform-y}
\end{equation}
uniformly in $t$.  No choice of the sign of the eigenvector is needed below, because only the product of the two endpoint coordinates occurs.

Multiplying \eqref{eq:scaled-end1}--\eqref{eq:scaled-end2} by $\sqrt n$ and using \eqref{eq:endpoint-small} gives
\[
\sqrt n\,x_t=\frac{(y_t)_1}{\mu_{j,t}}+O(n^{-1}),
\qquad
\sqrt n\,z_t=\frac{(y_t)_k}{\mu_{j,t}}+O(n^{-1}).
\]
Together with \eqref{eq:uniform-large-eigs} and \eqref{eq:uniform-y}, this proves the uniform limit
\begin{equation}
\sup_{t\in[0,1]}
\left|n x_tz_t-\frac{(w_j)_1(w_j)_k}{\theta_j^2}\right|
\longrightarrow0.
\label{eq:uniform-xz}
\end{equation}

Return to the unscaled eigenvalue
$\lambda_{j,t}(n)=n\mu_{j,t}(n)$ of $B_t(n,d_n)$.  Since it is simple, the Hellmann--Feynman formula applied to the $uv$ entry gives
\[
\frac{d}{dt}\lambda_{j,t}(n)=2x_tz_t.
\]
If $\theta_j>0$, then $\lambda_{j,t}(n)>0$ for all $t$ and all sufficiently large $n$, and \eqref{eq:uniform-large-eigs}--\eqref{eq:uniform-xz} yield, uniformly in $t$,
\begin{align}
\frac1{n^{p-2}}\frac{d}{dt}\lambda_{j,t}(n)^p
&=2p\left(\frac{\lambda_{j,t}(n)}n\right)^{p-1}
   \bigl(nx_tz_t\bigr)\notag\\
&\longrightarrow
2p\,\theta_j^{p-3}(w_j)_1(w_j)_k.
\label{eq:uniform-derivative}
\end{align}
If $\theta_j<0$, the corresponding large eigenvalue remains negative and contributes nothing to positive $p$-energy.  Because the convergence in \eqref{eq:uniform-derivative} is uniform on the compact interval $[0,1]$, integration in $t$ and summation over the positive eigenvalues of $Q$ give
\[
2p\sum_{\theta_j>0}\theta_j^{p-3}(w_j)_1(w_j)_k
=2p\,(Q_+^{p-3})_{1k}
\]
from the $k$ eigenvalues of order $n$.

It remains to justify that the two eigenvalues near the two zero eigenvalues of $S_n$ are negligible.  Let $\lambda=n\mu$ be either of these eigenvalues and let $(x,y,z)$ be a corresponding unit eigenvector.  Weyl's bound and \eqref{eq:uniform-matrix-pert} imply
\begin{equation}
\sup_{t\in[0,1]}|\mu|\longrightarrow0.
\label{eq:small-scaled}
\end{equation}
Put $\delta=\min_j|\theta_j|>0$.  For large $n$, uniformly in $t$ and for either stray eigenvalue,
\[
\|(Q_n-\mu I)^{-1}\|\le \frac2\delta.
\]
The unscaled core equation therefore gives the explicit estimate
\begin{equation}
\|y\|
\le \frac{2}{\delta\sqrt n}\sqrt{x^2+z^2}.
\label{eq:small-core-bound}
\end{equation}
Writing $q=\sqrt{x^2+z^2}$, normalization and
\eqref{eq:small-core-bound} give $q\ge1/2$ for all sufficiently large $n$.  The two endpoint equations, viewed as a two-dimensional vector equation, now imply
\[
|\lambda|q
\le \sqrt n\sqrt{y_1^2+y_k^2}+tq
\le \left(\frac2\delta+1\right)q.
\]
Consequently
\begin{equation}
|\lambda|\le \frac2\delta+1
\label{eq:stray-O1}
\end{equation}
for both stray eigenvalues, uniformly for $t\in[0,1]$ and all sufficiently large $n$.  Their total positive $p$-energy at $t=0$ and at $t=1$ is therefore $O(1)$, which is $o(n^{p-2})$ because $p>2$.  This completes the proof of \eqref{eq:limit}.
\end{proof}

\section{Noninteger exponents}
We first prove Theorem \ref{thm:main} when $p>2$ is not an integer. Put
\[
k=\lfloor p\rfloor\ge2,
\qquad
r=p-3.
\]
Then
\begin{equation}
k-3<r<k-2.
\label{eq:rwindow}
\end{equation}
Let $A=A(P_k)$ and choose a rational number
\begin{equation}
0<c<\frac1{\rho(A)}
=\frac1{2\cos(\pi/(k+1))}.
\label{eq:csmall}
\end{equation}
Then both
\[
Q=I+cA,
\qquad
M=I-cA
\]
are positive definite. Let
\[
D=\operatorname{diag}(1,-1,1,-1,\dots).
\]
Since $DAD=-A$, one has $Q=DMD$ and therefore, by functional calculus,
\begin{equation}
Q^r=DM^rD.
\label{eq:similarpower}
\end{equation}

Because $c\rho(A)<1$, the generalized binomial series converges absolutely in operator norm:
\begin{equation}
M^r=(I-cA)^r
=\sum_{j=0}^{\infty}\binom rj(-cA)^j.
\label{eq:binomial}
\end{equation}
The distance between the endpoints of $P_k$ is $k-1$, so
\[
(A^j)_{1k}=0\qquad(j<k-1),
\]
while $(A^{k-1})_{1k}=1$ and $(A^j)_{1k}\ge0$ for all $j$.

For every integer $j\ge k-1$, condition \eqref{eq:rwindow} implies that exactly $j-k+2$ factors in
\[
\binom rj=\frac{r(r-1)\cdots(r-j+1)}{j!}
\]
are negative. Hence
\[
\operatorname{sgn}\!\left(\binom rj(-1)^j\right)
=(-1)^{j-k+2+j}=(-1)^k.
\]
Thus every nonzero summand in the $(1,k)$ entry of \eqref{eq:binomial} has the same sign, and the first one is nonzero. Therefore
\begin{equation}
\operatorname{sgn}(M^r_{1k})=(-1)^k.
\label{eq:Msign}
\end{equation}
Since $D_{11}D_{kk}=(-1)^{k-1}$, equations \eqref{eq:similarpower} and \eqref{eq:Msign} give the strict sign
\begin{equation}
(Q^r)_{1k}<0.
\label{eq:Qnegative-nonint}
\end{equation}
As $Q$ is positive definite, $Q_+^r=Q^r$.

Write $c=a/b$ in lowest terms and take $n$ through multiples of $b$, with $d=cn$. Lemma \ref{lem:limit} and \eqref{eq:Qnegative-nonint} give
\[
\lim_{n\to\infty}
\frac{\mathcal E_p^+(H_{n,d,k}^{(1)})-\mathcal E_p^+(H_{n,d,k}^{(0)})}{n^{p-2}}
=2p(Q^r)_{1k}<0.
\]
Hence every sufficiently large such $n$ is a counterexample, and there are infinitely many.

\section{Integer exponents}
Now let $p=m\ge3$ be an integer, and put
\[
\vartheta=\frac{\pi}{m+1}.
\]
Choose a rational $c\in(0,1)$ such that
\begin{equation}
\frac1{2\cos\vartheta}<c,
\qquad
1-2c\cos(2\vartheta)>0.
\label{eq:cbig}
\end{equation}
Such a choice always exists. The first lower bound is below $1$; for $m=3$ the second inequality is automatic, while for $m\ge4$ the strict inequality $\cos\vartheta>\cos(2\vartheta)$ leaves a nonempty interval before the second spectral threshold.

The eigenvalues of
\[
Q=I_m+cA(P_m)
\]
are
\[
\theta_j=1+2c\cos\frac{j\pi}{m+1},
\qquad1\le j\le m.
\]
By \eqref{eq:cbig}, $\theta_m<0<\theta_{m-1}$. Thus $Q$ has exactly one negative eigenvalue and no zero eigenvalue.

Set $r=m-3$. Since the endpoints of $P_m$ have distance $m-1$ and $r<m-1$,
\begin{equation}
(Q^r)_{1m}=0.
\label{eq:locality}
\end{equation}
Let $w_m$ be the normalized sine eigenvector corresponding to $\theta_m$. Then
\[
w_m(1)w_m(m)
=\frac{2}{m+1}(-1)^{m+1}\sin^2\vartheta.
\]
Since $\theta_m<0$ and $r=m-3$,
\[
\theta_m^r w_m(1)w_m(m)>0.
\]
Using the spectral decomposition in \eqref{eq:locality}, and recalling that $\theta_m$ is the only negative eigenvalue, gives
\[
0=(Q_+^r)_{1m}+\theta_m^r w_m(1)w_m(m).
\]
Therefore
\begin{equation}
(Q_+^{m-3})_{1m}<0.
\label{eq:Qnegative-int}
\end{equation}
As in the noninteger case, take $n$ through multiples of the denominator of $c$ and put $d=cn$. Lemma \ref{lem:limit} yields
\[
\lim_{n\to\infty}
\frac{\mathcal E_m^+(H_{n,d,m}^{(1)})-\mathcal E_m^+(H_{n,d,m}^{(0)})}{n^{m-2}}
=2m(Q_+^{m-3})_{1m}<0.
\]
This proves Theorem \ref{thm:main} at every integer $m\ge3$, and together with the previous section completes its proof for every real $p>2$.

\section{An explicit $p=4$ certificate}
Although the asymptotic theorem already settles $p=4$, a concrete exact example is useful. Take four cliques $C_1,C_2,C_3,C_4\cong K_9$, join consecutive cliques completely, add vertices $u,v$, join $u$ to all of $C_1$ and $v$ to all of $C_4$, and omit $uv$. Call the resulting $38$-vertex graph $J_0$, and put $J_1=J_0+uv$.

An equitable-partition calculation gives
\begin{align}
\chi_{J_0}(x)
&=(x+1)^{32}
(x^3-25x^2+46x+153)
(x^3-7x^2-98x-9),\label{eq:J0-char}\\
\chi_{J_1}(x)
&=(x+1)^{34}(x+7)(x-14)(x^2-27x+98).
\label{eq:J1-char}
\end{align}
Let $-a$ be the negative root of the first cubic in \eqref{eq:J0-char}, and let $-b,-c$ be the two negative roots of the second, with $b>c>0$. Their relevant negative spectra are
\[
J_0:\quad -1\ (32\text{ times}),\ -a,-b,-c,
\qquad
J_1:\quad -1\ (34\text{ times}),\ -7.
\]
Direct rational evaluations give
\begin{equation}
a<\frac{421}{250},
\qquad
b<\frac{6939}{1000},
\qquad
c<\frac{93}{1000}.
\label{eq:abc-bounds}
\end{equation}
Indeed $a$ is the unique positive root of
$a^3+25a^2+46a-153$, while $b,c$ are the two positive roots of
$z^3+7z^2-98z+9$; the derivatives have fixed signs on the stated isolating intervals.

Adding $uv$ creates no new $4$-cycle and both endpoint degrees are $9$. From
\[
\operatorname{tr}A^4
=2|E|+4\sum_w\binom{d(w)}2+8C_4
\]
we obtain
\[
\operatorname{tr}A(J_1)^4-\operatorname{tr}A(J_0)^4=2+4(9+9)=74.
\]
Since total fourth energy is the fourth spectral moment,
\[
\mathcal E_4^+(J_1)-\mathcal E_4^+(J_0)
=a^4+b^4+c^4-2329.
\]
Using \eqref{eq:abc-bounds},
\begin{equation}
\mathcal E_4^+(J_1)-\mathcal E_4^+(J_0)
<-
\frac{1281106782111}{500000000000}<0.
\label{eq:p4-cert}
\end{equation}
Numerically the difference is approximately $-3.8556020556$.

\section{A $17$-vertex exact counterexample and its full failure interval}
Take $n=m=5$. Write $G_0=G_{5,5}^{(0)}$ and $G_1=G_{5,5}^{(1)}$. In this case
\begin{align*}
\det(xI-P_0)&=x^3-8x^2-39x+20=:f_0(x),\\
\det(xI-P_1)&=x^3-9x^2-31x+54=:f_1(x),\\
\det(xI-M_0)&=(x-5)(x+1),\\
\det(xI-M_1)&=x^2-3x-9.
\end{align*}
Hence
\begin{align}
\chi_{G_0}(x)&=(x+1)^{13}(x-5)f_0(x),\label{eq:char0}\\
\chi_{G_1}(x)&=(x+1)^{12}(x^2-3x-9)f_1(x).\label{eq:char1}
\end{align}
Because $P_0$ and $P_1$ are real symmetric, all roots of the two cubics are real. Their determinants are negative (the cubic constant terms are positive) while their traces are $8$ and $9$, respectively; hence each cubic has exactly one negative root. Write those roots as $-a$ and $-b$, respectively. Then
\begin{align}
a^3+8a^2-39a-20&=0,\label{eq:a}\\
b^3+9b^2-31b-54&=0.\label{eq:b}
\end{align}
The negative eigenvalue of $M_1$ has absolute value
\[
c=\frac{3(\sqrt5-1)}2.
\]

Because $N_{G_0}(u)\cap N_{G_0}(v)=\varnothing$, adding $uv$ creates no triangle. Therefore
\[
\operatorname{tr}A(G_1)^3=\operatorname{tr}A(G_0)^3,
\]
and
\begin{equation}
\Delta(3):=
\mathcal E_3^+(G_1)-\mathcal E_3^+(G_0)
=-55+27\sqrt5+b^3-a^3.
\label{eq:delta3}
\end{equation}
On $[3,\infty)$ the polynomials in \eqref{eq:a}--\eqref{eq:b} are strictly increasing, and direct evaluation at $3$ and $4$ places both $a$ and $b$ in $(3,4)$. Exact evaluations then give
\[
a>\frac{753}{200},
\qquad
b<\frac{454}{125},
\qquad
\sqrt5<\frac{2237}{1000}.
\]
Consequently
\begin{equation}
\Delta(3)
<-
\frac{59470157}{10^9}<0.
\label{eq:delta3cert}
\end{equation}
This alone gives a short exact $17$-vertex counterexample at the endpoint $p=3$ of the previously open range.

We now characterize all $p$ for which this same graph pair is a counterexample.

\begin{lemma}[Generalized Descartes rule]\label{lem:descartes}
Let
\[
F(s)=\sum_{i=1}^r c_i a_i^s,
\qquad
0<a_1<a_2<\cdots<a_r,
\qquad c_i\ne0.
\]
Then the number of distinct real zeros of $F$ is at most the number of sign changes in the sequence
$(c_1,\dots,c_r)$.
\end{lemma}

\begin{proof}
Set $x=e^s$ and $r_i=\log a_i$. Multiplying by the positive factor $x^{-r_1}$ reduces the claim to a generalized polynomial
\[
f(x)=c_1+\sum_{i=2}^r c_i x^{\rho_i},
\qquad
0<\rho_2<\cdots<\rho_r,
\qquad x>0.
\]
We induct on $r$. The derivative has coefficient signs $(c_2,\dots,c_r)$, since all $\rho_i$ are positive. If $c_1$ and $c_2$ have opposite signs, Rolle's theorem gives
\[
Z(f)\le Z(f')+1,
\]
matching the additional sign change at the front. If $c_1$ and $c_2$ have the same sign, then for sufficiently small $x>0$ both $f(x)$ and $f'(x)$ have that sign. If $f$ has a first positive zero, $f'$ must therefore have a zero at or before it; together with the Rolle zeros between consecutive zeros of $f$, this gives
\[
Z(f)\le Z(f').
\]
The induction hypothesis completes the proof.
\end{proof}

Let $x_0<X_0$ be the two positive roots of $f_0$, and $x_1<X_1$ the two positive roots of $f_1$. The positive eigenvalues are therefore
\[
\{x_0,5,X_0\}
\quad\text{for }G_0,
\qquad
\left\{x_1,\frac{3+3\sqrt5}{2},X_1\right\}
\quad\text{for }G_1.
\]
They satisfy the strict ordering
\begin{equation}
0<x_0<x_1<\frac{3+3\sqrt5}{2}<5<X_0<X_1.
\label{eq:ordering}
\end{equation}
Indeed, $f_0(0)>0>f_0(1)$, while $f_1(1)>0>f_1(2)$; also $(3+3\sqrt5)/2<5$. Next,
\[
f_0\left(\frac{56}{5}\right)=-\frac{1924}{125}<0,
\]
so $X_0>56/5$. Since
\[
f_1(x)-f_0(x)=-x^2+8x+34
\]
is decreasing for $x>4$ and is already negative at $56/5$, we have $f_1(X_0)<0$, hence $X_1>X_0$.

Thus
\begin{equation}
\Delta(p)
=-x_0^p+x_1^p
+\left(\frac{3+3\sqrt5}{2}\right)^p
-5^p-X_0^p+X_1^p.
\label{eq:exp-poly}
\end{equation}
The coefficient signs in increasing order of the bases are
\[
-,+,+,-,-,+,
\]
which have exactly three sign changes. Lemma \ref{lem:descartes} therefore implies that $\Delta$ has at most three distinct real zeros. One is the trivial zero
\[
\Delta(0)=0.
\]

We next certify signs at $p=2,3,4$. Since one edge is added,
\[
\operatorname{tr}A(G_1)^2-\operatorname{tr}A(G_0)^2=2.
\]
Using the negative spectra from \eqref{eq:char0}--\eqref{eq:char1},
\begin{equation}
\Delta(2)
=a^2-b^2+\frac{9\sqrt5-21}{2}.
\label{eq:delta2}
\end{equation}
The additional exact bound
\[
\sqrt5>\frac{559}{250}
\]
gives
\begin{equation}
\Delta(2)>
\frac{545801}{10^6}>0.
\label{eq:delta2cert}
\end{equation}
Equation \eqref{eq:delta3cert} gives $\Delta(3)<0$.

Finally, the standard closed-walk identity
\[
\operatorname{tr}A^4
=2|E|+4\sum_w\binom{d(w)}2+8C_4
\]
shows that
\[
\operatorname{tr}A(G_1)^4-\operatorname{tr}A(G_0)^4=42:
\]
the new edge contributes $2$, the two endpoint degrees are both $5$, and no new $4$-cycle is created. Hence
\begin{equation}
\Delta(4)=43+a^4-b^4-c^4.
\label{eq:delta4}
\end{equation}
Using the same rational bounds as above yields
\begin{equation}
\Delta(4)>
\frac{929118595792143}{16000000000000}>0.
\label{eq:delta4cert}
\end{equation}

We can now state the complete characterization.

\begin{theorem}\label{thm:fullinterval}
For the $17$-vertex pair $(G_0,G_1)$ there exist unique numbers
\[
\alpha\in(2,3),
\qquad
\beta\in(3,4)
\]
such that, for every $p\ge1$,
\[
\mathcal E_p^+(G_1)<\mathcal E_p^+(G_0)
\quad\Longleftrightarrow\quad
\alpha<p<\beta.
\]
Numerically,
\[
\alpha\approx2.4322519338217608578,
\qquad
\beta\approx3.0408359954191677907.
\]
\end{theorem}

\begin{proof}
Continuity together with \eqref{eq:delta2cert} and \eqref{eq:delta3cert} produces a zero $\alpha\in(2,3)$; continuity together with \eqref{eq:delta3cert} and \eqref{eq:delta4cert} produces a zero $\beta\in(3,4)$. Along with $p=0$, these are three distinct real zeros. By Lemma \ref{lem:descartes}, there can be no others. The signs at $2,3,4$ then force exactly the stated sign pattern on $p\ge1$.
\end{proof}

\section{Discussion}
Corollary \ref{cor:allp} gives a complete negative resolution of the spectral-radius-style monotonicity program above.  In particular, it disproves Conjecture~3 of Tang--Liu--Wang, arXiv v3 \cite{TangLiuWang2025v3}: there is no exponent $p\ge1$ for which positive $p$-energy is universally monotone under adding a missing edge.  Thus, although the spectral radius itself satisfies $\rho(G+uv)\ge\rho(G)$, none of the positive $p$-energies inherits a universal analogue. The mechanism is different at integer and noninteger exponents. Away from the integers, a fractional-power binomial series has a rigid endpoint sign dictated by path distance. At an integer, that leading term vanishes by locality, and moving the quotient across its first spectral threshold creates one negative mode whose removal from the positive spectral power restores a strict negative sign.

The $38$-vertex example shows that the failure at $p=4$ can be certified with elementary rational arithmetic. The $17$-vertex example gives a different perspective: for one fixed graph pair the failure occurs only on a bounded interval of exponents, even though the family construction produces counterexamples at every exponent by changing the graph.

Recent work proves path-minimality of total $p$-energy for every $p\ge2$ \cite{LiuTangPathEnergy2026} and path-minimality of positive $p$-energy in several regimes, including all connected graphs at $p=4$ \cite{LiuTang2026}.  Those results sit alongside recent lower bounds and structural inequalities for positive and negative $p$-energies \cite{AkbariEtAl2025,ChenWangZhang2026,TangElphickZhang2026,LiuTangZhang2026}.  Equation \eqref{eq:p4-cert} emphasizes that such global extremal statements do not imply local edge-addition monotonicity.

\section*{Acknowledgements}
OpenAI's ChatGPT was used as a computational and drafting aid in searching for the constructions, deriving the equitable-quotient reductions, checking symbolic algebra, and preparing verification code. All load-bearing inequalities and asymptotic arguments are stated explicitly in the manuscript, and the accompanying script independently reconstructs the finite examples and exact characteristic-polynomial certificates. The author is responsible for the mathematical content and final manuscript.


\begin{thebibliography}{99}
\bibitem{TangLiuWang2026}
Q.~Tang, Y.~Liu, and W.~Wang,
\newblock On the positive and negative $p$-energies of graphs under edge addition,
\newblock \emph{Discrete Appl. Math.} \textbf{388} (2026), 25--33.
\newblock doi:10.1016/j.dam.2026.02.045; arXiv:2410.09830.

\bibitem{TangLiuWang2025v3}
Q.~Tang, Y.~Liu, and W.~Wang,
\newblock On positive and negative $r$-th power energy of graphs with edge addition,
\newblock arXiv:2410.09830v3, revised 26 March 2025.

\bibitem{CioabaGuoSrivastava2021}
S.~M.~Cioab\u{a}, K.~Guo, and N.~Srivastava,
\newblock Spectral graph and hypergraph theory: connections and applications,
\newblock AIM Workshop Summary, American Institute of Mathematics, 2021.
\newblock \href{https://aimath.org/pastworkshops/spectralhypergraphrep.pdf}{Workshop summary}.

\bibitem{AbiadEtAl2023}
A.~Abiad, L.~de Lima, D.~N.~Desai, K.~Guo, L.~Hogben, and J.~Madrid,
\newblock Positive and negative square energies of graphs,
\newblock \emph{Electron. J. Linear Algebra} \textbf{39} (2023), 307--326.
\newblock doi:10.13001/ela.2023.7827.

\bibitem{ElphickEtAl2016}
C.~Elphick, M.~Farber, F.~Goldberg, and P.~Wocjan,
\newblock Conjectured bounds for the sum of squares of positive eigenvalues of a graph,
\newblock \emph{Discrete Math.} \textbf{339} (2016), 2215--2223.
\newblock doi:10.1016/j.disc.2016.01.021.

\bibitem{ElphickLinz2024}
C.~Elphick and W.~Linz,
\newblock Symmetry and asymmetry between positive and negative square energies of graphs,
\newblock \emph{Electron. J. Linear Algebra} \textbf{40} (2024), 418--432.
\newblock doi:10.13001/ela.2024.8447.

\bibitem{AkbariEtAl2025}
S.~Akbari, H.~Kumar, B.~Mohar, and S.~Pragada,
\newblock Vertex partitioning and $p$-energy of graphs,
\newblock \emph{Linear Algebra Appl.} \textbf{724} (2025), 96--107.
\newblock doi:10.1016/j.laa.2025.06.009.

\bibitem{LiuTangZhang2026}
Y.~Liu, Q.~Tang, and S.~Zhang,
\newblock The positive and negative square-energy conjecture,
\newblock arXiv:2607.18031 (2026).

\bibitem{ChenWangZhang2026}
Z.~Chen, Z.~Wang, and X.-D.~Zhang,
\newblock Positive and negative $3$-energies of graphs,
\newblock arXiv:2604.15656 (2026).

\bibitem{TangElphickZhang2026}
C.~Elphick, Q.~Tang, and S.~Zhang,
\newblock A spectral lower bound on chromatic numbers using $p$-energy,
\newblock \emph{European J. Combin.} \textbf{132} (2026), 104252.
\newblock doi:10.1016/j.ejc.2025.104252.

\bibitem{Nikiforov2012}
V.~Nikiforov,
\newblock Extremal norms of graphs and matrices,
\newblock \emph{J. Math. Sci.} \textbf{182} (2012), 164--174.
\newblock doi:10.1007/s10958-012-0737-z.

\bibitem{Nikiforov2016}
V.~Nikiforov,
\newblock Beyond graph energy: Norms of graphs and matrices,
\newblock \emph{Linear Algebra Appl.} \textbf{506} (2016), 82--138.
\newblock doi:10.1016/j.laa.2016.05.011.

\bibitem{ArizmendiGuerrero2023}
O.~Arizmendi and J.~Guerrero,
\newblock On the $p$-Schatten energy of bipartite graphs,
\newblock \emph{Acta Math. Hungar.} \textbf{169} (2023), 503--509.
\newblock doi:10.1007/s10474-023-01319-5.

\bibitem{LiuTangPathEnergy2026}
Y.~Liu and Q.~Tang,
\newblock Path-minimality of $p$-energy for connected graphs,
\newblock arXiv:2605.22730 (2026).

\bibitem{LiuTang2026}
Y.~Liu and Q.~Tang,
\newblock Path-minimality for positive $p$-energies, Laplacian-type spectra, and line graphs,
\newblock arXiv:2606.30996 (2026).
\newblock doi:10.48550/arXiv.2606.30996.

\bibitem{HornJohnson2013}
R.~A.~Horn and C.~R.~Johnson,
\newblock \emph{Matrix Analysis}, 2nd ed.,
\newblock Cambridge University Press, Cambridge, 2013.

\bibitem{CvetkovicEtAl2010}
D.~Cvetkovi\'c, P.~Rowlinson, and S.~Simi\'c,
\newblock \emph{An Introduction to the Theory of Graph Spectra},
\newblock Cambridge University Press, Cambridge, 2010.
\end{thebibliography}
\end{document}